\documentclass[11pt]{article}
\usepackage{latexsym}
\usepackage{bbm}
\usepackage{amscd}
\usepackage{amsmath}
\usepackage{amsfonts}
\usepackage{amssymb}
\usepackage{latexsym}
\usepackage[dvips]{graphicx}
\usepackage[colorlinks, citecolor=red]{hyperref}
\usepackage[marginal]{footmisc}
\usepackage[all]{xy}
\usepackage{verbatim}

\newtheorem{theorem}{\bf Theorem}[section]

\newtheorem{lemma}[theorem]{\bf Lemma}
\newtheorem{prop}[theorem]{\bf Proposition}

\newtheorem{remark}[theorem]{\bf Remark}
\newenvironment{proof}{\noindent{\em Proof:}}{\quad \hfill$\Box$\vspace{2ex}}

\def\no{\noindent}


\numberwithin{equation}{section}

\newenvironment{sequation}{\begin{equation}\small}{\end{equation}}
\newenvironment{seqnarray}{\begin{eqnarray}\small}{\end{eqnarray}}

\newcommand{\z}{\left}
\newcommand{\y}{\right}

\begin{document}

\thispagestyle{empty}
\begin{center}
{\LARGE \bf
The space of totally real flat minimal surfaces in the Quaternionic projective space $\mathbb{H}P^3$
\no
}
\end{center}

\begin{center}
Chuzi Duan
\footnote{
C. Duan
\\Center for Applied Mathematics and KL-AAGDM, Tianjin University, Tianjin, 300072, China
\\E-mail: duanchuzi@tju.edu.cn
}
and
Ling He
\footnote{
L. He (Corresponding author)
\\Center for Applied Mathematics and KL-AAGDM, Tianjin University, Tianjin, 300072, China
\\E-mail: heling@tju.edu.cn
}
\end{center}

\begin{center}
\parbox{12cm}
{\footnotesize{\bf ABSTRACT.}  We prove that the moduli space of all noncongruent linearly full totally real flat minimal immersions from the complex plane $\mathbb{C}$ into $\mathbb{H}P^3$ that do not lie in $\mathbb{C}P^3$ has three components, each of which is a manifold of real dimension $6$. As an application, we give a description of the moduli space of all noncongruent linearly full totally real flat minimal tori in $\mathbb{H}P^3$ that do not lie in $\mathbb{C}P^3$.
}
\end{center}

\no
{\bf{Keywords and Phrases.}} Twistor lift, minimal surfaces, totally real, moduli space.\\

\no
{\bf{Mathematics Subject Classification (2020).}} 53C42; 53C26; 58D10.

\section{Introduction}

In a theoretical physics language, we study two-dimensional sigma models with values in symmetric spaces. Din-Zakrzewski gave all solutions in the $\mathbb{C}P^{N-1}$ model and found the complexity of solutions in Grassmannian sigma models (cf. \cite{DZ1980},\cite{DZ1981}).
In mathematics, we look for harmonic maps from Riemann surface into symmetric spaces.
Eells-Wood in \cite{EW1983} described all harmonic two-spheres and harmonic two-tori of nonzero degree in $\mathbb{C}P^n$.
Chern-Wolfson in \cite{Chern-Wolfson1983} gave a moving frames interpretation, then in \cite{Chern-Wolfson1987} studied harmonic two-spheres in complex Grassmannians and proposed harmonic sequence by ``crossing'' construction.
Burstall-Wood in \cite{BF-WJ} developed a technique of analyzing harmonic maps from a Riemann surface into a complex Grassmannian by using ``diagrams'' to improve harmonic sequence theory.
More generally, Uhlenbeck in \cite{Uhlenbeck} found that any harmonic two-spheres to the unitary group or a complex Grassmannian can be obtained from a constant map by adding unitons, and showed that the moduli space of solutions is an algebraic variety, but did not give parametric description of this moduli space.
Based on Uhlenbeck's work, Bahy-El-Dien and Wood in \cite{BA-WJ} constructed explicitly all harmonic two-spheres in the quaternionic projective space $\mathbb{H}P^n$ by using twistor theory.
Through understanding harmonic map theories and studying its geometry, we obtained a series of classification and construction results about minimal two-spheres of constant Gauss curvature in $\mathbb{H}P^n$ (cf. \cite{HeJiao2014},\cite{HeJiao2015},\cite{HeJiao2015-2},\cite{FeiHe2017},\cite{ChenJiao2017},\cite{FeiPengXu2020},\cite{JiaoCui2021},\cite{ZhangJiao2021},\cite{JiaoXuXin2022}).
These results are important for studying minimal surfaces in symmetric spaces.

In order to further study the moduli space of minimal surfaces, we need some non-linear techniques.
Based on Penrose's idea in his twistor theory, Atiyah-Hitchin-Singer in \cite{AHS1978} used Atiyah-Singer index theorem to compute the dimension of the moduli space of self-dual irreducible connections.
In addition, twistor theory can be applied to study the space of harmonic two-spheres in compact symmetric spaces.
It is well known that the complex projective space $\mathbb{C}P^{2n+1}$ is the twistor space of $\mathbb{H}P^n$.
Loo in \cite{Loo1989} used twistor theory to show that the space of all harmonic maps of $S^2$ to $S^4$ of degree $d$ is known to be a connected complex algebraic variety of pure dimension $2d+4$.
Later, Kobak-Loo in \cite{KobakLoo1998} proved that the moduli space of quaternionic superminimal two-spheres in $\mathbb{H}P^n$ of degree $d$ is a connected quasi-projective variety of dimension $2nd+2n+2$.
However, the twistor space of spheres $S^{2n}~(n>2)$ is not biholomorphic to the complex projective spaces.
Fern\'{a}ndez in \cite{Fernandez} overcome this difficulty and showed that the space of harmonic maps is locally isomorphic to the space of holomorphic maps to $\mathbb{C}P^{n(n+1)/2}$, consequently proved that the dimension of the space of harmonic two-spheres in $S^{2n}$ of degree $\mathrm{d}$ is $2\mathrm{d}+n^2$.
Recently, Chi-Xie-Xu in \cite{CXX2021} used the singular-value decomposition theory to study the moduli space of noncongruent constantly curved minimal two-spheres in the complex hyperquadric $\mathcal{Q}_{n-1}$ which are also minimal in $\mathbb{C}P^n$.

It is natural to consider conformal minimal surfaces of higher genus in $\mathbb{H}P^n$.
Chi in \cite{Chi2000} applied deformation theory to study the dimension of the moduli space of superminimal surfaces of a fixed degree and conformal structure in $S^4$.
Generally, the method of integrable systems has been developed to find harmonic tori.
Burstall in \cite{Burstall1995} used such methods to construct all harmonic tori in spheres and complex projective spaces and showed that they are covered by primitive harmonic map of finite type.
Udagawa in \cite{U} combined the criterion that harmonic two-torus in a complex Grassmannian is of finite type with harmonic sequence theory to give a classification of all harmonic tori in $\mathbb{H}P^2$ and $\mathbb{H}P^3$, but didn't give the parameterization of the moduli space.

From the viewpoint of twistor theory, the existence of twistor lift is very important.
Fortunately, given a totally real isometric minimal immersion from a simply connected domain into $\mathbb{H}P^n$, there exists a totally real isometric horizontal minimal lift into $\mathbb{C}P^{2n+1}$ (cf. \cite{HeWang2005},\cite{HeZhou}).
He-Wang in \cite{HeWang2005} proved that the Veronese sequences in $\mathbb{R}P^{2m}$ ($2\leq 2m\leq n$) are the only totally real minimal two-spheres with constant Gauss curvature in $\mathbb{H}P^n$.
Recently, He-Zhou in \cite{HeZhou} considered totally real flat minimal surfaces in $\mathbb{H}P^n$ and find that the linearly full totally real flat minimal surfaces of isotropy order $n$ in $\mathbb{H}P^n$ lie in $\mathbb{C}P^n$, up to symplectic congruence.
In contrast to the two-sphere case, we ask whether other flat minimal surfaces with smaller isotropy order always lie in $\mathbb{C}P^n$.
The answer is no.
We can get many examples that do not lie in $\mathbb{C}P^n$.
Here we hope to study the space of all noncongruent linearly full totally real flat minimal surfaces in $\mathbb{H}P^n$ that do not lie in $\mathbb{C}P^n$.
Of course it is well known that the space of all noncongruent linearly full totally real flat minimal immersions from the complex plane $\mathbb{C}$ into $\mathbb{C}P^n$ is of real dimension $2(n-2)$ (cf. \cite{Liao}).

In this paper, we describe the moduli space of all noncongruent linearly full totally real flat minimal immersions from $\mathbb{C}$ into $\mathbb{H}P^3$ that do not lie in $\mathbb{C}P^3$ as follows (see Theorem \ref{moduli-space-C}).
\begin{theorem}\label{thm1}
    Let $\mathcal{M}_3(\mathbb{C})$ denote the moduli space of all noncongruent linearly full totally real flat minimal immersions from $\mathbb{C}$ into $\mathbb{H}P^3$ that do not lie in $\mathbb{C}P^3$, then
    $\mathcal{M}_3(\mathbb{C})$ has three components, each of which is a manifold of real dimension $6$ and intersects with two real hypersurfaces at
    \begin{small}
    \begin{equation*}
        \Gamma_3(\mathbb{C})=\z\{~ (\theta,r,w)\in \mathbb{R}^2\times\mathbb{C} ~\z|~\frac{\pi}{3}<\theta<\frac{\pi}{2},~0<r<\frac{1}{4\sin^2{\theta}}\y. ~\y\}.
    \end{equation*}
    \end{small}
\end{theorem}

As a direct application of Theorem \ref{thm1}, we can characterize the moduli space of all noncongruent linearly full totally real flat minimal tori in $\mathbb{H}P^3$ that do not lie in $\mathbb{C}P^3$ as follows (see Theorem \ref{moduli-space-tori}).
\begin{theorem}\label{thm2}
    Let $\mathcal{M}_3(\mathbb{T})$ denote the moduli space of all noncongruent linearly full totally real flat minimal tori in $\mathbb{H}P^3$ that do not lie in $\mathbb{C}P^3$, then
    $\mathcal{M}_3(\mathbb{T})$ has three components, each of which is a manifold of real dimension $4$ and intersects with two real hypersurfaces at
    \begin{small}
    \begin{equation*}
        \Gamma_3(\mathbb{T})=\z\{~ (\theta,r)\in \mathbb{R}^2 ~\z|~\frac{\pi}{3}<\theta<\frac{\pi}{2},~0<r<\frac{1}{4\sin^2{\theta}}\y. ~\y\}.
    \end{equation*}
    \end{small}
\end{theorem}

\begin{remark}
In a sense, Theorem \ref{thm2} gives an elegant description of the moduli space of harmonic two-tori in $\mathbb{H}P^3$ in the case of totally real flat.
A natural problem is how to extend this result to general target $\mathbb{H}P^n$ and general conformal harmonic two-tori.
We are optimistic about the former by promoting current methods.
\end{remark}

The paper is organized as follows.
In Sec.\ref{sec2}, we give some preliminaries.
In Sec.\ref{sec3}, we characterize all linearly full totally real flat minimal surfaces (local) in $\mathbb{H}P^3$ that do not lie in $\mathbb{C}P^3$ by discussing different isotropy orders (see Proposition \ref{prop1}), and give a description of the moduli space of all noncongruent linearly full totally real flat minimal immersions from $\mathbb{C}$ into $\mathbb{H}P^3$ that do not lie in $\mathbb{C}P^3$ (see Theorem \ref{moduli-space-C}).
In Sec.\ref{sec4}, we give the torus criterion (see Theorem \ref{torus-criterion}) and characterization of all linearly full totally real flat minimal tori in $\mathbb{H}P^3$ that do not lie in $\mathbb{C}P^3$ (see Theorem \ref{moduli-space-tori}).
At last, we discuss the case of isotropy order $2$ by the harmonic sequence (see Theorem \ref{isotropy-order-2}).

{\bf{Acknowledgments}}~
This work is supported by National Key R\&D Program of China No. 2022YFA1006600 and NSF in China Nos. 12071352, 12071338.

\section{Preliminaries}
\label{sec2}

Let $\mathbb{C}^{2n+2}$ be a $(2n+2)$-dimensional complex linear space with standard Hermitian inner product $\langle,\rangle$ defined by $\langle z,w \rangle=\sum\limits_{i=1}^{2n+2} z_i\bar{w}_i$, where $z=(z_1,\cdots,z_{2n+2})^T,w=(w_1,\cdots,w_{2n+2})^T$, and ~$\bar{}$~ denotes complex conjugation.
Let $\mathbb{H}$ be the division ring of quaternions, i.e.
\begin{small}
    \[
        \mathbb{H}=\z\{ a+b\texttt{i}+c\texttt{j}+d\texttt{k} \mid a,b,c,d\in\mathbb{R},~\texttt{i}^2=\texttt{j}^2=\texttt{k}^2=\texttt{i}\texttt{j}\texttt{k}=-1 \y\}.
    \]
\end{small}
We find that $\mathbb{H}^{n+1}=\z\{ (q_1,\cdots,q_{n+1})^T \mid q_k\in\mathbb{H}, k=1,\cdots,n+1 \y\}$ is a quaternion linear space under the following operations:
\begin{small}
    \begin{seqnarray}
        (p_1,\cdots,p_{n+1})^T+(q_1,\cdots,q_{n+1})^T&=&(p_1+q_1,\cdots,p_{n+1}+q_{n+1})^T\nonumber,\\
        (q_1,\cdots,q_{n+1})^T\cdot q&=&(q_1q,\cdots,q_{n+1}q)^T\nonumber.
    \end{seqnarray}
\end{small}
Since $a+b\texttt{i}+c\texttt{j}+d\texttt{k}=(a+b\texttt{i})+(c+d\texttt{i})\texttt{j}$, then we have a natural identification of $\mathbb{H}$ with $\mathbb{C}^2$, and thus $\mathbb{H}^{n+1}$ with $\mathbb{C}^{2n+2}$.
Notice that $\texttt{j}(z_1+z_2\texttt{j})=-\bar{z}_2+\bar{z}_1\texttt{j}$, then left multiplication by $\texttt{j}$ induce a conjugate linear map on $\mathbb{C}^{2n+2}$, also denoted by $\texttt{j}$, i.e.
\begin{small}
    \[
        \texttt{j}(z_1,z_2,\cdots,z_{2n+1},z_{2n+2})^T
        =
        (-\bar{z}_2,\bar{z}_1,\cdots,-\bar{z}_{2n+2},\bar{z}_{2n+1})^T.
    \]
\end{small}
Let $\mathbb{H}P^n$ be the set of all 1-dimensional subspaces of $\mathbb{H}^{n+1}$ with Fubini-Study metric.
Let $G(2,2n+2)$ be the Grassmann manifold of all complex 2-dimensional subspaces of $\mathbb{C}^{2n+2}$, then we can regard $\mathbb{H}P^n$ as the totally geodesic submanifold of $G(2,2n+2)$ as follows:
\begin{small}
    \[
        \mathbb{H}P^n
        =
        \z\{ V\in G(2,2n+2) \mid \texttt{j}V=V \y\}.
    \]
\end{small}

If $\varphi$ is a harmonic map from Riemann surface $M$ into $\mathbb{H}P^n$, then we can get its harmonic sequence $\{ \underline{\varphi}_{j} \}_{j\in \mathbb{N}}$ in $G(2,2n+2)$ as follows (cf. \cite{BA-WJ} and \cite{BF-WJ}):
\begin{sequation}
    \cdots
    \stackrel{A''_{\varphi_{-1}}}{\longleftarrow} \underline{\varphi}_{-1}
    \stackrel{A''_{\varphi_0}}{\longleftarrow} \underline{\varphi}_0=\underline{\varphi}
    \stackrel{A'_{\varphi_0}}{\longrightarrow} \underline{\varphi}_{1}
    \stackrel{A'_{\varphi_1}}{\longrightarrow} \underline{\varphi}_{2}
    \stackrel{A'_{\varphi_2}}{\longrightarrow} \cdots
    \stackrel{A'_{\varphi_{m-1}}}{\longrightarrow} \underline{\varphi}_{m}
    \stackrel{A'_{\varphi_m}}{\longrightarrow} \underline{\varphi}_{m+1}
    \stackrel{A'_{\varphi_{m+1}}}{\longrightarrow} \cdots,
\label{eq:34}
\end{sequation}
where $\underline{\varphi}_{-k}={\texttt{j}}\underline{\varphi}_{k}$ are $2$-dimensional harmonic subbundles of the trivial bundle $M\times \mathbb{C}^{2n+2}$.
Here
\[
    A'_{\varphi_0}(v)=\pi_{\varphi_0^{\perp}(\partial_z v)},
    ~A''_{\varphi_0}(v)=\pi_{\varphi_0^{\perp}(\partial_{\bar{z}} v)}
\]
for $v\in C^{\infty}(\underline{\varphi}_0)$, where $\pi_{\varphi_0^{\perp}}$ and $C^{\infty}(\underline{\varphi}_0)$ denote the orthogonal projection and the vector space of smooth sections of $\underline{\varphi}_0$ respectively.
We say that the {\it isotropy order} of $\varphi$ is $\mathrm{r}$ if $\underline{\varphi}_{0} \perp \underline{\varphi}_{i}$ for $i=1,\cdots,\mathrm{r}$ and $\underline{\varphi}_{\mathrm{r}+1}$ is not perpendicular to $\underline{\varphi}_{0}$.
In particular, if $\mathrm{r}=\infty$, then $\varphi$ is said to be {\it strongly isotropic}.

Applying the globally defined, non-degenerated four-form on $\mathbb{H}P^n$, we can define the {\it quaternionic K\"ahler angle} $\alpha$ with respect to $\varphi$ (cf. \cite{HeZhou}).
It gives a measure of the failure of $\varphi$ to be a totally complex map or a totally real map.
If $\alpha$ is identically equal to $\frac{\pi}{2}$ on $M$, then $\varphi$ is said to be {\it totally real}.

\section{Totally real flat minimal surfaces}
\label{sec3}

Suppose that $(M,ds_M^2)$ is a simply connected domain in the complex plane $\mathbb{C}$ with the flat metric $ds_M^2=2dzd\bar{z}$.
We consider the linearly full totally real isometric minimal immersion from $M$ into $\mathbb{H}P^n$.
The complex projective space $\mathbb{C}P^{2n+1}$ is the twistor space of $\mathbb{H}P^n$.
The twistor map $t:\mathbb{C}P^{2n+1} \rightarrow \mathbb{H}P^n$, is given by
$
    t\left([z_1,z_2,\cdots,z_{2n+1},z_{2n+2}]\right)
    =[z_1+z_2 \texttt{j},\cdots,z_{2n+1}+z_{2n+2}\texttt{j}].
$
Then $t$ is a Riemann submersion and the horizontal distribution is given by
$
    \sum\limits_{i=1}^{n+1}\left(z_{2i-1}dz_{2i}-z_{2i}dz_{2i-1}\right)=0.
$
Fortunately, we have the following result (cf. Theorem 3.4 in \cite{HeZhou}).
\begin{lemma}\label{lem3-1}
    Let $\varphi:M\to \mathbb{H}P^n$ be a linearly full totally real isometric minimal immersion, then there exists a totally real isometric horizontal minimal lift $[s]:M \to \mathbb{C}P^{2n+1}$.
\end{lemma}
In Lemma \ref{lem3-1}, $[s]$ may not be linearly full in $\mathbb{C}P^{2n+1}$, then it lies in $\mathbb{C}P^{m} \subset \mathbb{C}P^{2n+1}$ for $n\leq m \leq 2n+1$, where $m\geq n$ by using that $\varphi$ is linearly full in $\mathbb{H}P^n$.
In terms of the complex projective space, we have the following result (cf. Theorem 3.1 in \cite{Liao}).
\begin{lemma}
    If $f:M \rightarrow \mathbb{C}P^n$ is a linearly full totally real isometric minimal immersion, then up to unitary equivalence, $f=\left[V_0^{(n)}\right]$, where
    \begin{sequation}
        V_0^{(n)}(z)=
        \begin{pmatrix}
            e^{a_0z-\overline{a}_0\overline{z}}\xi^0,
            &e^{a_1z-\overline{a}_1\overline{z}}\xi^1,
            &\cdots,
            &e^{a_nz-\overline{a}_n\overline{z}}\xi^n
        \end{pmatrix}^T,
        \label{eq3-13}
    \end{sequation}
    and $a_k=e^{\texttt{i}\theta_k},~\xi^k=\sqrt{r_k}~(r_k>0)$ for $k=0,1,\cdots,n$, satisfying $0=\theta_0<\theta_1<\cdots <\theta_n<2\pi,
    ~r_0+r_1+\cdots+r_n=1$.
    \label{lem3-2}
\end{lemma}

Using Lemma \ref{lem3-2}, we know that $s$ is given by \eqref{eq3-13}, up to $U(2n+2)$. But the isometry group of $\mathbb{H}P^n$ is $Sp(n+1)$, which is a subgroup of $U(2n+2)$.
In order to give the explicit expression of $\varphi$, we need to characterize $s$, up to $Sp(n+1)$.
Now we discuss totally real flat minimal surfaces in $\mathbb{H}P^3$ by different isotropy orders $\mathrm{r}$.
For $\mathrm{r}=3$, we have the following result (cf. Theorem 4.2 in \cite{HeZhou}).
\begin{prop}
    Let $\varphi:M \rightarrow \mathbb{H}P^3$ be a linearly full totally real isometric minimal immersion of isotropy order $3$, then up to $Sp(4)$, $\varphi$ lies in $\mathbb{C}P^3~(\subset \mathbb{H}P^3)$ given by
    \begin{sequation}\label{sec4+00}
    \varphi=
    \begin{bmatrix}
        e^{a_0z-\overline{a}_0\overline{z}},
        &e^{a_1z-\overline{a}_1\overline{z}},
        &e^{a_2z-\overline{a}_2\overline{z}},
        &e^{a_{3}z-\overline{a}_{3}\overline{z}}
    \end{bmatrix}^T,
    \end{sequation}
    where $a_k=e^{\texttt{i}\frac{2k\pi}{4}}~(k=0,1,2,3)$ (the Clifford solution in $\mathbb{C}P^3$), or $a_k=e^{\texttt{i}\frac{k\pi}{4}}~(k=0,1,2,3)$.
    \label{thm3}
\end{prop}
For $\mathrm{r}=1$ and $\mathrm{r}=2$, we give the characterization of minimal surfaces as follows.
\begin{prop}
    \label{prop1}
    Let $\varphi:M \rightarrow \mathbb{H}P^3$ be a linearly full totally real isometric minimal immersion that does not lie in $\mathbb{C}P^3$, then $\varphi$ only have three types, and the isotropy order of $\varphi$ can be $1$ or $2$.
    If the isotropy order is at least $1$, then up to $Sp(4)$, $\varphi$ of each type is uniquely determined by five parameters, including four real and one complex, which satisfy certain constraints;
    if the isotropy order is $2$, then up to $Sp(4)$, $\varphi$ of each type is uniquely determined by only three parameters $\theta,r,w$ in $\mathbb{R}^2 \times \mathbb{C}$, which satisfy
    \[
        \frac{\pi}{3}<\theta<\frac{\pi}{2},
        ~0<r<\frac{1}{4\sin^2\theta},
        ~w\in \mathbb{C}.
    \]
\end{prop}
\begin{proof}
We will prove this proposition in three steps.
Firstly, we will find all $\varphi$ that lie in $\mathbb{H}P^3$ but do not lie in $\mathbb{C}P^3$ based on whether $a_j$ are paired; next, we will study $\varphi$ with isotropy order at least $1$; finally, we will study $\varphi$ with isotropy order $2$.

{\it Step 1}. Let $\varphi:M \to \mathbb{H}P^3$ be a linearly full totally real isometric minimal immersion.
Applying Lemma \ref{lem3-1}, we know that there exists a totally real isometric horizontal minimal lift $\underline{f}=[s]:M \to \mathbb{C}P^7$ such that $\underline{\varphi}=\underline{f} \oplus \texttt{j}\underline{f}$.
Then it follows from Lemma \ref{lem3-2} that there exists a unitary matrix $U \in U(8)$ such that $s=UV_0^{(m)}$, where $V_0^{(m)}$ is a vector in $\mathbb{C}^8$ as follows,
\[
    V_0^{(m)}(z)=
    \begin{pmatrix}
         e^{a_0z-\overline{a}_0\overline{z}}\xi^0,
        &e^{a_1z-\overline{a}_1\overline{z}}\xi^1,
        &\cdots,
        &e^{a_mz-\overline{a}_m\overline{z}}\xi^m,
        &0,
        &\cdots,
        &0
    \end{pmatrix}^T.
\]
Since $\varphi$ is linearly full in $\mathbb{H}P^3$, we know that $3 \leq m \leq 7$.
Let $W=U^TJU=(w_{ij})$, where $J$ is a block diagonal matrix of order $8$ as follows,
\[
    J=diag
    \z\{
        \begin{pmatrix}
            0 & -1\\
            1 & 0
        \end{pmatrix},
        \begin{pmatrix}
            0 & -1\\
            1 & 0
        \end{pmatrix},
        \begin{pmatrix}
            0 & -1\\
            1 & 0
        \end{pmatrix},
        \begin{pmatrix}
            0 & -1\\
            1 & 0
        \end{pmatrix}
    \y\},
\]
then $W$ is an anti-symmetric unitary matrix.
By Proposition 3.5 in \cite{HeZhou}, we have $\z\langle \partial_z^k s,~\texttt{j}s \y\rangle=\z\langle \partial_{\bar{z}}^k s,~\texttt{j}s \y\rangle=0~(k=1,2)$, then for $k=1,2$,
\begin{small}
    \begin{eqnarray}
        \sum_{i,j=0}^m \left(w_{ij}\xi^i\xi^ja_j^k e^{(a_i+a_j)z-(\overline{a}_i+\overline{a}_j)\overline{z}}\right)
        =
        \sum_{i,j=0}^m \left(w_{ij}\xi^i\xi^j(-\bar{a}_j)^k e^{(a_i+a_j)z-(\overline{a}_i+\overline{a}_j)\overline{z}}\right)
        =0.
    \label{w-eq}
    \end{eqnarray}
\end{small}
For a given $W$, we can define a set
\[
    G_W=\z\{ U\in U(8) \mid U^TJU=W \y\},
\]
then it is easy to check that $Sp(4) \cdot G_W=G_W$, which means that we will obtain $\varphi$ up to $Sp(4)$ as long as we determine $W$ and find some $U\in G_W$.

Next, we discuss $W$ in three cases based on whether the $a_j$ are paired.

\textbf{Case I:} $\forall ~i,j=0,1,\cdots,m,~a_i+a_j\neq 0$.
In this case, we have $w_{ij}=0$ by \eqref{w-eq}, which shows that
\[
    W=
    \begin{pmatrix}
        0_{(m+1)\times(m+1)} & *\\
        * & *\\
    \end{pmatrix}.
\]
Since $W$ is an anti-symmetric unitary matrix, then $m=3$.
We can choose $U\in G_W$ as follows:
\begin{small}
    \[
        U=
        \begin{pmatrix}
            1&0&0&0&0      &0      &0      &0\\
            0&0&0&0&-w_{04}&-w_{05}&-w_{06}&-w_{07}\\
            0&1&0&0&0      &0      &0      &0\\
            0&0&0&0&-w_{14}&-w_{15}&-w_{16}&-w_{17}\\
            0&0&1&0&0      &0      &0      &0\\
            0&0&0&0&-w_{24}&-w_{25}&-w_{26}&-w_{27}\\
            0&0&0&1&0      &0      &0      &0\\
            0&0&0&0&-w_{34}&-w_{35}&-w_{36}&-w_{37}\\
        \end{pmatrix}.
    \]
\end{small}
Then we get the horizontal lift of $\varphi$
\begin{small}
    \[
        s(z)
        =
        UV_0^{(3)}(z)
        =
        \begin{pmatrix}
            e^{a_0z-\bar{a}_0\bar{z}}\xi^0,
            &0,
            &e^{a_1z-\bar{a}_1\bar{z}}\xi^0,
            &0,
            &e^{a_2z-\bar{a}_2\bar{z}}\xi^0,
            &0,
            &e^{a_3z-\bar{a}_3\bar{z}}\xi^0,
            &0
        \end{pmatrix}^T,
    \]
\end{small}
which implies
\[
    \varphi
    =
    \begin{bmatrix}
        e^{a_0z-\bar{a}_0\bar{z}}\xi^0,
        &e^{a_1z-\bar{a}_1\bar{z}}\xi^0,
        &e^{a_2z-\bar{a}_2\bar{z}}\xi^0,
        &e^{a_3z-\bar{a}_3\bar{z}}\xi^0
    \end{bmatrix}^T.
\]
This lies in $\mathbb{C}P^3$.

\textbf{Case II:} $\forall ~i,~\exists ~j~s.t.~a_i+a_j=0$.
In this case, $m$ must be odd.
Since $3 \leq m \leq 7$, we only have three possibilities: $m=3,~m=5,~m=7$.

If $m=3$, then the same argument as the first case shows that
\begin{small}
    \[
        W=
        \begin{pmatrix}
            0_{4\times 4}&*\\
            *            &*
        \end{pmatrix},
    \]
\end{small}
which implies that $\varphi$ lies in $\mathbb{C}P^3$.

If $m=5$, then $a_0+a_3=a_1+a_4=a_2+a_5=0$.
From \eqref{w-eq} it follows that $w_{ij}=0$ for $0\leq i,j \leq 5$ and $|i-j|\neq 3$.
Moreover we have
\begin{small}
    \[
    \z\{
        \begin{aligned}
            w_{03}\xi^0\xi^3a_0+w_{14}\xi^1\xi^4a_1+w_{25}\xi^2\xi^5a_2=0,\\
            w_{03}\xi^0\xi^3\bar{a}_0+w_{14}\xi^1\xi^4\bar{a}_1+w_{25}\xi^2\xi^5\bar{a}_2=0.
        \end{aligned}
    \y.
    \]
\end{small}
By solving this linear system, we obtain
\begin{small}
    \[
    \z\{
        \begin{aligned}
            w_{14}&=w_{03}\cdot\frac{\xi^0\xi^3}{\xi^1\xi^4}
                          \cdot\frac{\sin\theta_2}{\sin(\theta_1-\theta_2)},\\
            w_{25}&=w_{03}\cdot\frac{\xi^0\xi^3}{\xi^2\xi^5}
                          \cdot\frac{\sin\theta_1}{\sin(\theta_2-\theta_1)}.
        \end{aligned}
    \y.
    \]
\end{small}
So
\begin{small}
    \[
        W=
        \begin{pmatrix}
            0&0&0&w_{03}&0&0&w_{06}&w_{07}\\
            0&0&0&0&w_{14}&0&w_{16}&w_{17}\\
            0&0&0&0&0&w_{25}&w_{26}&w_{27}\\
            -w_{03}&0&0&0&0&0&w_{36}&w_{37}\\
            0&-w_{14}&0&0&0&0&w_{46}&w_{47}\\
            0&0&-w_{25}&0&0&0&w_{56}&w_{57}\\
            -w_{06}&-w_{16}&-w_{26}&-w_{36}&-w_{46}&-w_{56}&w_{66}&w_{67}\\
            -w_{07}&-w_{17}&-w_{27}&-w_{37}&-w_{47}&-w_{57}&-w_{67}&w_{77}
        \end{pmatrix}.
    \]
\end{small}
Since $W$ is an anti-symmetric unitary matrix, we only have four possibilities:
\begin{enumerate}
    \item
        $
            \begin{pmatrix}
                w_{06}&w_{16}&w_{26}&w_{36}&w_{46}&w_{56}\\
                w_{07}&w_{17}&w_{27}&w_{37}&w_{47}&w_{57}
            \end{pmatrix}^T
            =
            \begin{pmatrix}
                A&0&0&-\bar{B}&0&0\\
                B&0&0&\bar{A} &0&0
            \end{pmatrix}^T,
        $
    \item
        $
            \begin{pmatrix}
                w_{06}&w_{16}&w_{26}&w_{36}&w_{46}&w_{56}\\
                w_{07}&w_{17}&w_{27}&w_{37}&w_{47}&w_{57}
            \end{pmatrix}^T
            =
            \begin{pmatrix}
                0&A&0&0&-\bar{B}&0\\
                0&B&0&0&\bar{A} &0
            \end{pmatrix}^T,
        $
    \item
        $
            \begin{pmatrix}
                w_{06}&w_{16}&w_{26}&w_{36}&w_{46}&w_{56}\\
                w_{07}&w_{17}&w_{27}&w_{37}&w_{47}&w_{57}
            \end{pmatrix}^T
            =
            \begin{pmatrix}
                0&0&A&0&0&-\bar{B}\\
                0&0&B&0&0&\bar{A}
            \end{pmatrix}^T,
        $
    \item
        $
            \begin{pmatrix}
                w_{06}&w_{16}&w_{26}&w_{36}&w_{46}&w_{56}\\
                w_{07}&w_{17}&w_{27}&w_{37}&w_{47}&w_{57}
            \end{pmatrix}^T
            =
            \begin{pmatrix}
                0&0&0&0&0&0\\
                0&0&0&0&0&0
            \end{pmatrix}^T,
        $
\end{enumerate}
where $A,~B\in\mathbb{C}$ such that $W\in U(8)$.
In the last situation, it is easy to find that $\varphi$ is not linearly full in $\mathbb{H}P^3$.
In the first situation, we can choose $U$ as follows:
\begin{small}
    \[
        U=
        \begin{pmatrix}
            1&0&0&0&0&0&0&0\\
            0&0&0&-w_{03}&0&0&-A&-B\\
            0&1&0&0&0&0&0&0\\
            0&0&0&0&-w_{03}\cdot\frac{\xi^0\xi^3}{\xi^1\xi^4}\cdot\frac{\sin\theta_2}{\sin(\theta_1-\theta_2)}&0&0&0\\
            0&0&1&0&0&0&0&0\\
            0&0&0&0&0&-w_{03}\cdot\frac{\xi^0\xi^3}{\xi^2\xi^5}\cdot\frac{\sin\theta_1}{\sin(\theta_2-\theta_1)}&0&0\\
            0&0&0&a&0&0&*&*\\
            0&0&0&b&0&0&*&*
        \end{pmatrix},
    \]
\end{small}
where $|a|^2+|b|^2+|w_{03}|^2=|w_{14}|^2=|w_{25}|^2=1$.
Then we obtain
\begin{small}
    \begin{sequation}\label{first}
        \varphi=
        \begin{bmatrix}
            (\xi^0-w_{03}\cdot \xi^3\cdot \texttt{j})\cdot e^{z-\bar{z}}\\
            (\xi^1-w_{03}\cdot\frac{\xi^0\xi^3}{\xi^1}\cdot\frac{\sin\theta_2}{\sin(\theta_1-\theta_2)}\cdot \texttt{j})\cdot e^{a_1z-\bar{a}_1\bar{z}}\\
            (\xi^2-w_{03}\cdot\frac{\xi^0\xi^3}{\xi^2}\cdot\frac{\sin\theta_1}{\sin(\theta_2-\theta_1)}\cdot \texttt{j})\cdot e^{a_2z-\bar{a}_2\bar{z}}\\
            \xi^3 e^{\bar{a}_0\bar{z}-a_0z}(a+b\texttt{j})
        \end{bmatrix}.
    \end{sequation}
\end{small}
By applying similar arguments to the second and third situations, we obtain
\begin{small}
    \begin{sequation}\label{second}
        \varphi=
        \begin{bmatrix}
            (\xi^0-w_{03}\cdot \xi^3\cdot \texttt{j})\cdot e^{z-\bar{z}}\\
            (\xi^1-w_{03}\cdot\frac{\xi^0\xi^3}{\xi^1}\cdot\frac{\sin\theta_2}{\sin(\theta_1-\theta_2)}\cdot \texttt{j})\cdot e^{a_1z-\bar{a}_1\bar{z}}\\
            (\xi^2-w_{03}\cdot\frac{\xi^0\xi^3}{\xi^2}\cdot\frac{\sin\theta_1}{\sin(\theta_2-\theta_1)}\cdot \texttt{j})\cdot e^{a_2z-\bar{a}_2\bar{z}}\\
            \xi^4 e^{\bar{a}_1\bar{z}-a_1z}(a+b\texttt{j})
        \end{bmatrix},
    \end{sequation}
\end{small}
where $|a|^2+|b|^2+|w_{14}|^2=|w_{03}|^2=|w_{25}|^2=1$, and
\begin{small}
    \begin{sequation}\label{third}
        \varphi=
        \begin{bmatrix}
            (\xi^0-w_{03}\cdot \xi^3\cdot \texttt{j})\cdot e^{z-\bar{z}}\\
            (\xi^1-w_{03}\cdot\frac{\xi^0\xi^3}{\xi^1}\cdot\frac{\sin\theta_2}{\sin(\theta_1-\theta_2)}\cdot \texttt{j})\cdot e^{a_1z-\bar{a}_1\bar{z}}\\
            (\xi^2-w_{03}\cdot\frac{\xi^0\xi^3}{\xi^2}\cdot\frac{\sin\theta_1}{\sin(\theta_2-\theta_1)}\cdot \texttt{j})\cdot e^{a_2z-\bar{a}_2\bar{z}}\\
            \xi^5 e^{\bar{a}_2\bar{z}-a_2z}(a+b\texttt{j})
        \end{bmatrix},
    \end{sequation}
\end{small}
where $|a|^2+|b|^2+|w_{25}|^2=|w_{03}|^2=|w_{14}|^2=1$ respectively.
When $b \neq 0$, none of these three types of mappings lie in $\mathbb{C}P^3$ up to $Sp(4)$.
They all truly lie in $\mathbb{H}P^3$.

If $m=7$, then $a_0+a_4=a_1+a_5=a_2+a_6=a_3+a_7=0$.
From \eqref{w-eq} it follows that $w_{ij}=0$ for $0\leq i,j \leq 7$ and $|i-j|\neq 4$.
Moreover we have
\begin{small}
    \[
    \z\{
        \begin{aligned}
            w_{04}\xi^0\xi^4a_0+w_{15}\xi^1\xi^5a_1+w_{26}\xi^2\xi^6a_2+w_{37}\xi^3\xi^7a_3=0,\\
            w_{04}\xi^0\xi^4\bar{a}_0+w_{15}\xi^1\xi^5\bar{a}_1+w_{26}\xi^2\xi^6\bar{a}_2+w_{37}\xi^3\xi^7\bar{a}_3=0.
        \end{aligned}
    \y.
    \]
\end{small}
By solving this linear system, we obtain
\begin{small}
    \[
    \z\{
        \begin{aligned}
            w_{26}&=\frac{w_{04}\xi^0\xi^4\sin\theta_3+w_{15}\xi^1\xi^5\sin(\theta_3-\theta_1)}
                        {\xi^2\xi^6\sin(\theta_2-\theta_3)},\\
            w_{37}&=\frac{w_{04}\xi^0\xi^4\sin\theta_2+w_{15}\xi^1\xi^5\sin(\theta_2-\theta_1)}
                        {\xi^3\xi^7\sin(\theta_3-\theta_2)}.
        \end{aligned}
    \y.
    \]
\end{small}
So
\begin{small}
    \[
        W=
        \begin{pmatrix}
            0      &0      &0      &0      &w_{04}&     0&0     &0     \\
            0      &0      &0      &0      &0     &w_{15}&0     &0     \\
            0      &0      &0      &0      &0     &0     &w_{26}&0     \\
            0      &0      &0      &0      &0     &0     &0     &w_{37}\\
            -w_{04}&0      &0      &0      &0     &0     &0     &0     \\
            0      &-w_{15}&0      &0      &0     &0     &0     &0     \\
            0      &0      &-w_{26}&0      &0     &0     &0     &0     \\
            0      &0      &0      &-w_{37}&0     &0     &0     &0
        \end{pmatrix},
    \]
\end{small}
where $|w_{26}|^2=|w_{37}|^2=1$.
One of the elements $U$ in $G_W$ is given as follows:
\begin{small}
    \[
        U=
        \begin{pmatrix}
            1&0&0&0&0      &0      &0      &0      \\
            0&0&0&0&-w_{04}&0      &0      &0      \\
            0&1&0&0&0      &0      &0      &0      \\
            0&0&0&0&0      &-w_{15}&0      &0      \\
            0&0&1&0&0      &0      &0      &0      \\
            0&0&0&0&0      &0      &-w_{26}&0      \\
            0&0&0&1&0      &0      &0      &0      \\
            0&0&0&0&0      &0      &0      &-w_{37}
        \end{pmatrix}.
    \]
\end{small}
Then we can easily compute the horizontal lift $s$ and finally get $\varphi$ as follows:
\begin{small}
    \[
        \varphi=
        \begin{bmatrix}
            (\xi^0-w_{04}\xi^4\cdot \texttt{j}) e^{z-\bar{z}}~~~~~\\
            (\xi^1-w_{15}\xi^5\cdot \texttt{j}) e^{a_1z-\bar{a}_1\bar{z}}\\
            (\xi^2-\frac{w_{04}\xi^0\xi^4\sin\theta_3+w_{15}\xi^1\xi^5\sin(\theta_3-\theta_1)}{\xi^2\sin(\theta_2-\theta_3)}\cdot \texttt{j})         e^{a_2z-\bar{a}_2\bar{z}}\\
            (\xi^3-\frac{w_{04}\xi^0\xi^4\sin\theta_2+w_{15}\xi^1\xi^5\sin(\theta_2-\theta_1)}{\xi^3\sin(\theta_3-\theta_2)}\cdot  \texttt{j})        e^{a_3z-\bar{a}_3\bar{z}}
        \end{bmatrix}.
    \]
\end{small}
This lies in $\mathbb{C}P^3$ up to $Sp(4)$.

\textbf{Case III:} $\exists ~i~s.t.~\forall ~j,~a_i+a_j \neq 0$ and $\exists~k,l~s.t.~a_k+a_l=0$.
In this case, we have the following possibilities:
\begin{small}
    \[
    \z\{
        \begin{aligned}
            m&=3\Rightarrow 2~unpaired, 2~paired; \\
            m&=4\Rightarrow 3~unpaired, 2~paired; 1~unpaired, 4~paired;\\
            m&=5\Rightarrow 4~unpaired, 2~paired; 2~unpaired, 4~paired;\\
            m&=6\Rightarrow 5~unpaired, 2~paired; 3~unpaired, 4~paired; 1~unpaired, 6~paired;\\
            m&=7\Rightarrow 6~unpaired, 2~paired; 4~unpaired, 4~paired; 2~unpaired, 6~paired.
        \end{aligned}
    \y.
    \]
\end{small}

If $m=3$ and $2$ paired $2$ unpaired, then without loss of generality, we may assume that $a_2+a_3=0,~\forall~j,~a_0+a_j \neq 0,~a_1+a_j \neq 0$.
From \eqref{w-eq} we obtain
\begin{small}
    \[
        W=
        \begin{pmatrix}
            0_{4\times 4}&*\\
            *            &*
        \end{pmatrix},
    \]
\end{small}
which implies that $\varphi$ lies in $\mathbb{C}P^3$.
Applying the same method as described above, we can analyze other possibilities and obtain that either $W$ degenerates or $\varphi$ lies in $\mathbb{C}P^3$ up to $Sp(4)$.

Now, we have found all linearly full totally real flat minimal surfaces in $\mathbb{H}P^3$ that do not lie in $\mathbb{C}P^3$.
They only have three types, namely those belonging to \eqref{first}, \eqref{second} or \eqref{third}.

{\it Step 2.} We will study \eqref{first}, \eqref{second} and \eqref{third} with isotropy order at least $1$.
Then we have $\sum\limits_{j=1}^m r_j=1$ and $\sum\limits_{j=1}^m a_jr_j=0$.
Moreover we have $m=5$ and $a_0+a_3=a_1+a_4=a_2+a_5=0$.
If we put $a_j=c_j+\texttt{i}s_j,~c_j=\cos{\theta_j},~s_j=\sin{\theta_j},~(j=0,1,\cdots,m)$, then we have the following system of equations
\begin{small}
    \begin{sequation}\label{2-eq}
        \begin{cases}
            r_0+r_3+r_1+r_4+r_2+r_5=1,\\
            (r_0-r_3)+c_1(r_1-r_4)+c_2(r_2-r_5)=0,\\
            s_1(r_1-r_4)+s_2(r_2-r_5)=0.
        \end{cases}
    \end{sequation}
\end{small}
Let $\square=s_1c_2-c_1s_2$, then \eqref{2-eq} is equivalent to the following system of equations
\begin{small}
    \begin{sequation}\label{2-eq-first}
        \begin{cases}
            r_1+r_4+r_0+r_3=1-(r_2+r_5),\\
            r_0-r_3=-\frac{\square}{s_1}\cdot(r_2-r_5),\\
            r_1-r_4=-\frac{s_2}{s_1}\cdot(r_2-r_5).
        \end{cases}
    \end{sequation}
\end{small}
From \eqref{2-eq-first}, we can express $r_0,r_1,r_3$ by $r_2,r_4,r_5,\theta_1,\theta_2$.
For \eqref{first} we have $|a|^2+|b|^2+|w_{03}|^2=|w_{14}|^2=|w_{25}|^2=1$, i.e.
\begin{small}
    \begin{sequation}\label{first-mofang=1}
        |a|^2+|b|^2+|w_{03}|^2
        =
        \z|w_{03}\cdot\frac{\xi^0\xi^3}{\xi^1\xi^4}\cdot\frac{\sin\theta_2}{\sin(\theta_1-\theta_2)}\y|^2
        =
        \z|w_{03}\cdot\frac{\xi^0\xi^3}{\xi^2\xi^5}\cdot\frac{\sin\theta_1}{\sin(\theta_2-\theta_1)}\y|^2
        =1.
    \end{sequation}
\end{small}
Then we can get $s_1^2r_1r_4=s_2^2r_2r_5$ from the above.
After expressing $r_0,r_1,r_3$ by $r_2,r_4,r_5,\theta_1,\theta_2$, we obtain a quadratic equation in $r_4$ with parameters as follows:
\begin{small}
    \[
        s_1^2\cdot r_4\cdot\z( r_4-\frac{s_2}{s_1}\cdot(r_2-r_5) \y)=s_2^2r_2r_5.
    \]
\end{small}
Since $r_4>0$, then a straightforward calculation shows $r_4=\frac{s_2r_2}{s_1}$.
By virtue of this, together with \eqref{2-eq-first} we can express $r_0,r_3,r_1,r_4$ by $r_2,r_5,\theta_1,\theta_2$ as follows:
\begin{small}
    \begin{sequation}
        \begin{cases}
            r_4=\frac{s_2r_2}{s_1},\\
            r_1=\frac{s_2r_5}{s_1},\\
            r_0=\frac{1}{2}\z( 1-r_2-r_5-\frac{s_2+\square}{s_1}r_2-\frac{s_2-\square}{s_1}r_5 \y),\\
            r_3=\frac{1}{2}\z( 1-r_2-r_5-\frac{s_2+\square}{s_1}r_5-\frac{s_2-\square}{s_1}r_2 \y).\\
        \end{cases}
    \end{sequation}
\end{small}
Next, we determine the constrains on our parameters $r_2,r_5,\theta_1,\theta_2$.
Since they satisfy the following conditions:
\begin{small}
    \begin{sequation}\label{guideline}
        \begin{cases}
            0<|a|^2+|b|^2<1,\\
            r_0,r_1,\cdots,r_5\in (0,1),\\
            0=\theta_0<\theta_1<\cdots<\theta_5<2\pi,
        \end{cases}
    \end{sequation}
\end{small}
then for \eqref{first} of isotropy order at least $1$, we have $0<\theta_1<\theta_2<\pi$ and $r_2,r_5>0$.
Thus we naturally have $r_1,r_4>0$.
From $r_1+r_4+r_2+r_5<1$ and $r_0,r_3>0$, we obtain $(1+\frac{s_2+\square}{s_1})r_2+(1+\frac{s_2-\square}{s_1})r_5<1$ and $(1+\frac{s_2+\square}{s_1})r_5+(1+\frac{s_2-\square}{s_1})r_2<1$.
Moreover we get $\frac{r_2r_5}{r_0r_3}\cdot\frac{\square^2}{s_1^2}<1$.
In all, we have the following constraints:
\begin{small}
    \[
    \z\{
        \begin{aligned}
            &0<\theta_1<\theta_2<\pi,r_2,r_5>0,\\
            &(1+\frac{s_2+\square}{s_1})r_p+(1+\frac{s_2-\square}{s_1})r_q<1,~(p,q)\in \{(2,5),(5,2)\},\\
            &\square^2r_2r_5<s_1^2r_0r_3.
        \end{aligned}
    \y.
    \]
\end{small}
For $a+b\texttt{j}$ in \eqref{first}, we can put $w=\frac{a}{b}$ and up to $Sp(4)$ the forth component of \eqref{first} will be
\begin{small}
    \[
        |b|\cdot\sqrt{r_3}\cdot e^{\bar{a}_0\bar{z}-a_0z}(w+\texttt{j}).
    \]
\end{small}
From \eqref{first-mofang=1} it is easy to see that $|b|$ can be expressed by $r_2,r_5,\theta_1,\theta_2$ and $w$.
If the isotropy order of \eqref{first} is at least $1$, then up to $Sp(4)$, $\varphi$ depends on five parameters, including four real and one complex.

For \eqref{second}, we employ a similar approach to derive the following equations
\begin{small}
    \begin{sequation}
        \begin{cases}
            r_2=-\frac{s_1r_0}{\square},\\
            r_5=-\frac{s_1r_3}{\square},\\
            r_1=\frac{1}{2}\z( 1-r_0-r_3+\frac{s_1+s_2}{\square}r_0+\frac{s_1-s_2}{\square}r_3 \y),\\
            r_4=\frac{1}{2}\z( 1-r_0-r_3+\frac{s_1+s_2}{\square}r_3+\frac{s_1-s_2}{\square}r_0 \y),\\
        \end{cases}
    \end{sequation}
\end{small}
subject to the following constraints
\begin{small}
    \[
    \z\{
        \begin{aligned}
            &0<\theta_1<\theta_2<\pi,r_0,r_3>0,\\
            &(1-\frac{s_1+s_2}{\square})r_p+(1-\frac{s_1-s_2}{\square})r_q<1,~(p,q)\in \{(0,3),(3,0)\},\\
            &s_2^2r_0r_3<\square^2r_1r_4.
        \end{aligned}
    \y.
    \]
\end{small}
For \eqref{third}, we obtain
\begin{small}
    \begin{sequation}
        \begin{cases}
            r_4=-\frac{s_2r_0}{\square},\\
            r_1=-\frac{s_2r_3}{\square},\\
            r_2=\frac{1}{2} \z( 1-r_0-r_3-\frac{s_1-s_2}{\square}r_0+\frac{s_1+s_2}{\square}r_3 \y),\\
            r_5=\frac{1}{2} \z( 1-r_0-r_3-\frac{s_1-s_2}{\square}r_3+\frac{s_1+s_2}{\square}r_0 \y),\\
        \end{cases}
    \end{sequation}
\end{small}
subject to
\begin{small}
    \[
    \z\{
        \begin{aligned}
            &0<\theta_1<\theta_2<\pi,r_0,r_3>0,\\
            &(1+\frac{s_1-s_2}{\square})r_p+(1-\frac{s_1+s_2}{\square})r_q<1,~(p,q)\in \{(0,3),(3,0)\},\\
            &s_1^2r_0r_3<\square^2r_2r_5.
        \end{aligned}
    \y.
    \]
\end{small}
From the derivation process of $\varphi$ we know that these surfaces are noncongruent for different parameters, each type is noncongruent too.
Then up to $Sp(4)$, $\varphi$ of each type can always be uniquely determined by five parameters, including four real and one complex.

{\it Step 3.} If the isotropy order of \eqref{first} is $2$, then we have $\sum\limits_{j=0}^m a_j^2r_j=0$.
Next, we compute $\sum\limits_{j=0}^m a_j^2r_j$ for \eqref{first}:
\begin{small}
    \[
        \begin{aligned}
            \sum_{j=0}^5 a_j^2r_j
            &=(r_0+r_3)+a_1^2(r_1+r_4)+a_2^2(r_2+r_5)\\
            &=\z( 1-\frac{s_2r_2}{s_1}-\frac{s_2r_5}{s_1}-r_2-r_5 \y)+a_1^2\z( \frac{s_2r_2}{s_1}+\frac{s_2r_5}{s_1} \y)+a_2^2\z( r_2+r_5 \y)\\
            &=(r_2+r_5)\cdot\z[ \z(\frac{1}{r_2+r_5}-(\frac{s_2}{s_1}+1)+\frac{(2c_1^2-1)s_2+(2c_2^2-1)s_1}{s_1}\y)+\texttt{i}\frac{2s_1c_1s_2+2s_2c_2s_1}{s_1} \y].
        \end{aligned}
    \]
\end{small}
So
\begin{small}
    \[
        \mathbf{Re}(\sum_{j=0}^5 a_j^2r_j)=\mathbf{Im}(\sum_{j=0}^5 a_j^2r_j)=0,
    \]
\end{small}
i.e.
\begin{small}
    \[
        \frac{1}{r_2+r_5}-(\frac{s_2}{s_1}+1)+\frac{(2c_1^2-1)s_2+(2c_2^2-1)s_1}{s_1}=\frac{2s_1c_1s_2+2s_2c_2s_1}{s_1}=0.
    \]
\end{small}
Easy computation shows that the above equation is equivalent to that $\theta_1+\theta_2=\pi$ and $r_2+r_5=\frac{1}{4s_1^2}$, then we can express $\theta_2,r_0,r_3,r_1,r_4,r_5$ by $\theta_1,r_2$, i.e.
\begin{small}
    \begin{sequation}\label{tongjie}
        \begin{cases}
            \theta_2=\pi-\theta_1,\\
            r_4=\frac{s_2r_2}{s_1}=r_2,\\
            r_5=\frac{1}{4s_1^2}-r_2,\\
            r_1=\frac{s_2r_5}{s_1}=r_5=\frac{1}{4s_1^2}-r_2,\\
            r_0=\frac{1}{2}\z( 1-r_2-r_5-\frac{s_2+\square}{s_1}r_2-\frac{s_2-\square}{s_1}r_5 \y)
                =\frac{1}{2}-\frac{1}{4s_1^2}-\frac{c_1}{4s_1^2}+2c_1r_2,\\
            r_3=\frac{1}{2}\z( 1-r_2-r_5-\frac{s_2+\square}{s_1}r_5-\frac{s_2-\square}{s_1}r_2 \y)
                =\frac{1}{2}-\frac{1}{4s_1^2}+\frac{c_1}{4s_1^2}-2c_1r_2.\\
        \end{cases}
    \end{sequation}
\end{small}
Since $\theta_1<\theta_2$, then $0<\theta_1<\frac{\pi}{2}$.
Because $r_2,r_5>0$, then $0<r_2<\frac{1}{4s_1^2}$.
From $r_0,r_3>0$, we can obtain $\frac{2c_1-1}{8c_1(1-c_1)} < r_2 < \frac{2c_1+1}{8c_1(1+c_1)}$.
Next, solving inequality $\frac{r_2r_5}{r_0r_3}\cdot\frac{\square^2}{s_1^2} < 1$ yields $4s_1^2 > 3$, i.e., $\frac{\pi}{3}<\theta_1<\frac{\pi}{2}$, which implies that $\frac{2c_1-1}{8c_1(1-c_1)}<0$ and $\frac{1}{4s_1^2} < \frac{2c_1+1}{8c_1(1+c_1)}$.
So our constraint region reduces to
\[
    \left\{~
        \frac{\pi}{3}<\theta_1<\frac{\pi}{2},
        0<r_2<\frac{1}{4s_1^2},
        w\in \mathbb{C}~
    \right\}.
\]
Moreover up to $Sp(4)$, \eqref{first} can be simplified into
\begin{small}
    \begin{seqnarray}\label{new-first}
        \varphi
        &=&
        \begin{bmatrix}
            (\xi^0-w_{03}\cdot \xi^3\cdot \texttt{j})\cdot e^{z-\bar{z}}\\
            (\xi^1-w_{03}\cdot\frac{\xi^0\xi^3}{\xi^1}\cdot\frac{\sin\theta_2}{\sin(\theta_1-\theta_2)}\cdot \texttt{j})\cdot e^{a_1z-\bar{a}_1\bar{z}}\\
            (\xi^2-w_{03}\cdot\frac{\xi^0\xi^3}{\xi^2}\cdot\frac{\sin\theta_1}{\sin(\theta_2-\theta_1)}\cdot \texttt{j})\cdot e^{a_2z-\bar{a}_2\bar{z}}\\
            \xi^3 e^{a_3z-\bar{a}_3\bar{z}}(a+b\texttt{j})
        \end{bmatrix}
        =
        \begin{bmatrix}
            \sqrt{r_0+|w_{03}|^2r_3}\cdot e^{z-\overline{z}}\\
            \sqrt{r_1+r_2} \cdot e^{a_1z-\overline{a}_1\overline{z}}\\
            \sqrt{r_1+r_2}\cdot e^{-\overline{a}_1z+a_1\overline{z}}\\
            \sqrt{r_3}\cdot e^{-z+\overline{z}}(a+b\texttt{j})
        \end{bmatrix}\nonumber\\
        &=&
        \begin{bmatrix}
            \sqrt{\frac{4c_1^2r_2(\frac{1}{4s_1^2}-r_2)}{\frac{1}{2}-\frac{1}{4s_1^2}-\frac{c_1}{4s_1^2}+2c_1r_2}+\frac{1}{2}-\frac{1}{4s_1^2}-\frac{c_1}{4s_1^2}+2c_1r_2}\cdot e^{z-\overline{z}}\\
            \frac{1}{2s_1}\cdot e^{a_1z-\overline{a}_1\overline{z}}\\
            \frac{1}{2s_1}\cdot e^{-\overline{a}_1z+a_1\overline{z}}\\
            |b|\cdot\sqrt{\frac{1}{2}-\frac{1}{4s_1^2}+\frac{c_1}{4s_1^2}-2c_1r_2}\cdot e^{-z+\overline{z}}(w+\texttt{j})
        \end{bmatrix}.
    \end{seqnarray}
\end{small}
Applying similar computation to \eqref{second} and \eqref{third} gives us
\begin{small}
    \begin{seqnarray}\label{new-second}
        \varphi
        &=&
        \begin{bmatrix}
            \frac{1}{2s_1}\cdot e^{z-\overline{z}}\\
            \sqrt{\frac{4c_1^2r_0(\frac{1}{4s_1^2}-r_0)}{\frac{1}{2}-\frac{1}{4s_1^2}-2c_1r_0+\frac{c_1}{4s_1^2}}+\frac{1}{2}-\frac{1}{4s_1^2}-2c_1r_0+\frac{c_1}{4s_1^2}}\cdot e^{a_1z-\overline{a}_1\overline{z}}\\
            \frac{1}{2s_1}\cdot e^{a_1^2z-\overline{a}_1^2\overline{z}}\\
            |b|\cdot\sqrt{\frac{1}{2}-\frac{1}{4s_1^2}+2c_1r_0-\frac{c_1}{4s_1^2}} \cdot e^{\overline{a}_1\overline{z}-a_1z}(w+\texttt{j})
        \end{bmatrix},
        \quad
        \z\{
            \begin{aligned}
                \frac{\pi}{3}<&\theta_1<\frac{\pi}{2},\\
                0<&r_0<\frac{1}{4s_1^2},\\
                w&\in \mathbb{C},
            \end{aligned}
        \y.
    \end{seqnarray}
\end{small}
and
\begin{small}
    \begin{seqnarray}\label{new-third}
        \varphi
        &=&
        \begin{bmatrix}
            \frac{1}{2s_2}\cdot e^{z-\overline{z}}\\
            \frac{1}{2s_2}\cdot e^{\bar{a}_2^2\bar{z}-a_2^2z}\\
            \sqrt{\frac{4c_2^2r_0(\frac{1}{4s_2^2}-r_0)}{\frac{1}{2}-\frac{1}{4s_2^2}-2c_2r_0+\frac{c_2}{4s_2^2}}+\frac{1}{2}-\frac{1}{4s_2^2}-2c_2r_0+\frac{c_2}{4s_2^2}}\cdot e^{a_2z-\overline{a}_2\overline{z}}\\
            |b|\cdot\sqrt{\frac{1}{2}-\frac{1}{4s_2^2}+2c_2r_0-\frac{c_2}{4s_2^2}}\cdot e^{\overline{a}_2\overline{z}-a_2z}(w+\texttt{j})
        \end{bmatrix},
        \quad
        \z\{
            \begin{aligned}
                \frac{\pi}{2}<&\theta_2<\frac{2\pi}{3},\\
                0<&r_0<\frac{1}{4s_2^2},\\
                w&\in \mathbb{C},
            \end{aligned}
        \y.
    \end{seqnarray}
\end{small}
respectively.
In order to unify the constraint regions, we put $\theta=\theta_1,r=r_2$ for \eqref{new-first}, put $\theta=\theta_1,r=r_0$ for \eqref{new-second} and put $\theta=\pi-\theta_2,r=r_0$ for \eqref{new-third}.
Then \eqref{new-third} will become
\begin{small}
    \begin{seqnarray}\label{new-third-theta changed}
        \varphi
        &=&
        \begin{bmatrix}
            \frac{1}{2s_2}\cdot e^{z-\overline{z}}\\
            \frac{1}{2s_2}\cdot e^{\frac{\bar{z}}{\bar{a}_2^2}-\frac{z}{a_2^2}}\\
            \sqrt{\frac{4c_2^2r_0(\frac{1}{4s_2^2}-r_0)}{\frac{1}{2}-\frac{1}{4s_2^2}+2c_2r_0-\frac{c_2}{4s_2^2}}+\frac{1}{2}-\frac{1}{4s_2^2}+2c_2r_0-\frac{c_2}{4s_2^2}}\cdot e^{\frac{\bar{z}}{\bar{a}_2}-\frac{z}{a_2}}\\
            |b|\cdot\sqrt{\frac{1}{2}-\frac{1}{4s_2^2}-2c_2r_0+\frac{c_2}{4s_2^2}}\cdot e^{\frac{z}{a_2}-\frac{\bar{z}}{\bar{a}_2}}(w+\texttt{j})
        \end{bmatrix},
        \quad
        \z\{
            \begin{aligned}
                \frac{\pi}{3}<&\theta<\frac{\pi}{2},\\
                0<&r<\frac{1}{4\sin^2\theta},\\
                w&\in \mathbb{C}.
            \end{aligned}
        \y.
    \end{seqnarray}
\end{small}
It is obvious that \eqref{new-first}, \eqref{new-second} and \eqref{new-third} (i.e., \eqref{new-third-theta changed}) are noncongruent for different values of the parameters, but these expressions have similar form, i.e.
\begin{small}
    \begin{sequation}\label{varphi-general-form}
        \varphi(z)=
        \begin{bmatrix}
            \eta^0 e^{a_0z-\overline{a}_0\overline{z}}\\
            \eta^1 e^{a_1z-\overline{a}_1\overline{z}}\\
            \eta^2 e^{a_2z-\overline{a}_2\overline{z}}\\
            \eta^3 e^{\overline{a}_k\overline{z}-a_kz} \cdot (w+\texttt{j})
        \end{bmatrix}
        \quad
        (k=0,~1,~2),
    \end{sequation}
\end{small}
where $\eta^j\in\mathbb{R^+}~(j=0,~1,~2,~3),~w\in\mathbb{C}$.
Then up to $Sp(4)$, $\varphi$ of each type is uniquely determined by only three parameters $\theta,r,w$ in $\mathbb{R}^2\times \mathbb{C}$, which satisfy
\begin{small}
    \[
        \frac{\pi}{3}<\theta<\frac{\pi}{2},
        ~0<r<\frac{1}{4\sin^2\theta},
        ~w\in \mathbb{C}.
    \]
\end{small}
Summing up, we complete our proofs.
\end{proof}

Now we describe the moduli space of the totally real flat minimal immersions from $\mathbb{C}$ into $\mathbb{H}P^3$ that do not lie in $\mathbb{C}P^3$.
Applying Proposition \ref{prop1}, we can obtain the following theorem.
\begin{theorem}\label{moduli-space-C}
    Let $\mathcal{M}_3(\mathbb{C})$ denote the moduli space of all noncongruent linearly full totally real flat minimal immersions from $\mathbb{C}$ into $\mathbb{H}P^3$ that do not lie in $\mathbb{C}P^3$, then $\mathcal{M}_3(\mathbb{C})$ has three components, each of which is a manifold of real dimension $6$ and intersects with two real hypersurfaces at
    \begin{small}
    \begin{equation*}
        \Gamma_3(\mathbb{C})=\z\{~ (\theta,r,w)\in \mathbb{R}^2\times\mathbb{C} ~\z|~ \frac{\pi}{3}<\theta<\frac{\pi}{2},~
        0<r<\frac{1}{4\sin^2{\theta}}\y.~ \y\}.
    \end{equation*}
    \end{small}
\end{theorem}
\begin{proof}
    Taking $M=\mathbb{C}$ in Proposition \ref{prop1}, we know that all linearly full totally real flat minimal immersions from $\mathbb{C}$ into $\mathbb{H}P^3$ that do not lie in $\mathbb{C}P^3$ only have three types, and if the isotropy order is $2$ then the parameter spaces of these surfaces are all $\Gamma_3(\mathbb{C})$.
    Since $\Gamma_3(\mathbb{C}) \neq \varnothing$, then we have $(\Gamma_3(\mathbb{C})\subset)\mathcal{M}_3(\mathbb{C}) \neq  \varnothing$.
    Notice that these three types are noncongruent for different parameters, so $\mathcal{M}_3(\mathbb{C})$ have three components.
    From \eqref{guideline} we know that each component is an open subset of the Euclidean space, and thus a manifold.
    From Proposition \ref{prop1}, we know that $\varphi$ of each type can be uniquely determined by five parameters, including four real and one complex.
    So each component of $\mathcal{M}_3(\mathbb{C})$ is a manifold of real dimension $6$.
    When the isotropy order is $2$, two equations have been added to the parameter spaces of these three types.
    It makes each one reduce to $\Gamma_3(\mathbb{C})$.
    Hence each component of $\mathcal{M}_3(\mathbb{C})$ intersects with two real hypersurfaces at $\Gamma_3(\mathbb{C})$.
    Thus we have finished our proof.
\end{proof}
\begin{remark}
    Figure \ref{GC} gives the constraint region of $(\theta,r)$ in $\Gamma_3(\mathbb{C})$.
    Here the minimal surfaces corresponding to the points on the boundary lie in $\mathbb{C}P^3$.
    \begin{figure}[htb]
        \center{\includegraphics[width=10cm]{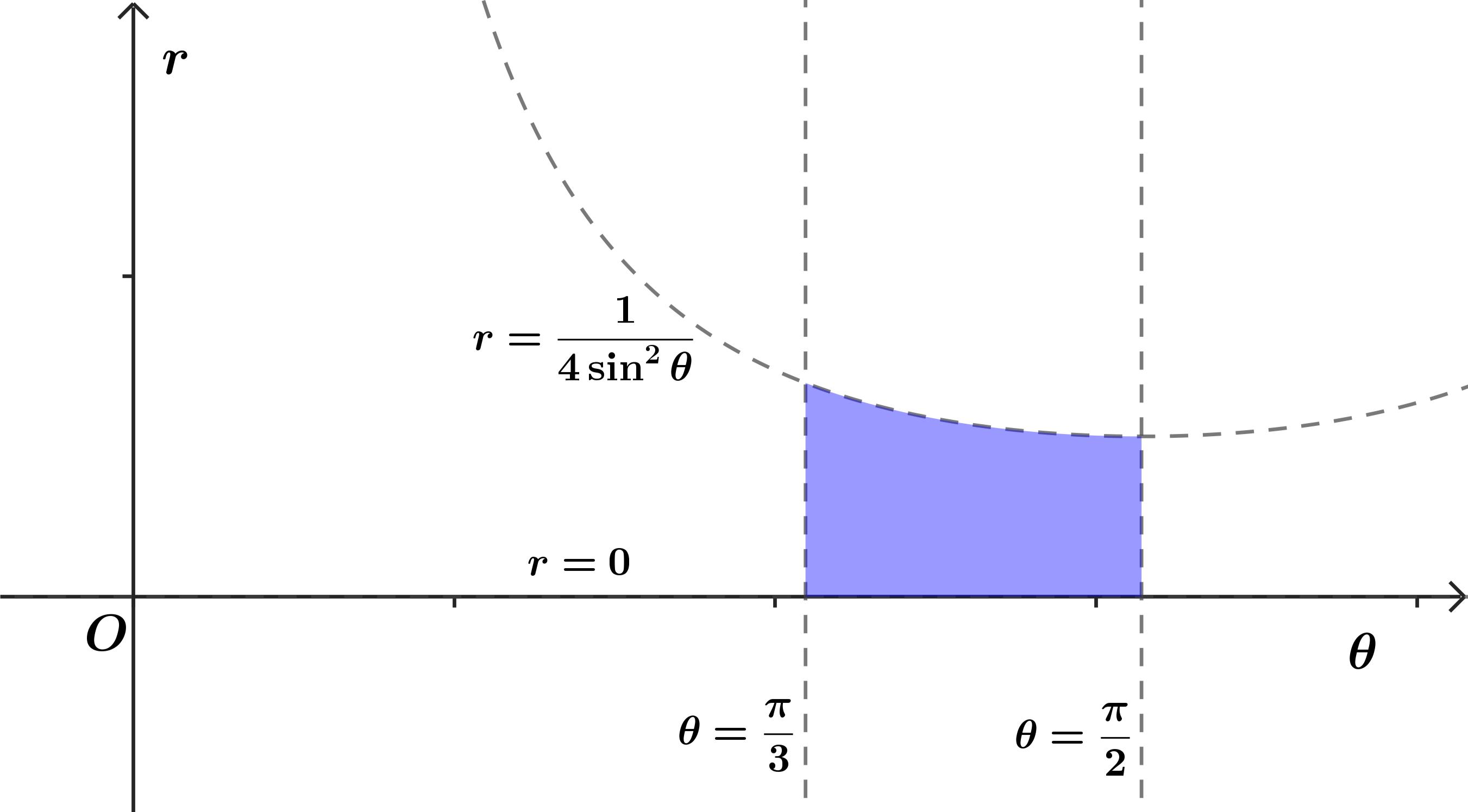}}
        \caption{$\frac{\pi}{3}<\theta<\frac{\pi}{2},0<r<\frac{1}{4\sin^2{\theta}}$}
        \label{GC}
    \end{figure}
\end{remark}

\begin{remark}
    For totally real flat minimal surfaces in $\mathbb{H}P^3$, a unit increase in the isotropy order decreases the dimension of the moduli space by $2$ real parameters.
    Is this phenomenon universal in $\mathbb{H}P^n$?
\end{remark}

\section{Tori}
\label{sec4}

Next, we will consider when $\varphi$ can descend to the torus.
Notice that although the isotropy order of \eqref{first}, \eqref{second} and \eqref{third} may be $1$ or $2$, their expressions have similar form, i.e., \eqref{varphi-general-form}.
Therefore we discuss the corresponding torus criterion for \eqref{varphi-general-form} when $k=0,~1,~2$.
For points $z,~z'\in\mathbb{C}$, $\varphi(z)=\varphi(z')$ if and only if
\begin{small}
    \[
        \begin{pmatrix}
            \eta^0 e^{a_0z-\overline{a}_0\overline{z}}\\
            \eta^1 e^{a_1z-\overline{a}_1\overline{z}}\\
            \eta^2 e^{a_2z-\overline{a}_2\overline{z}}\\
            \eta^3 e^{\overline{a}_k\overline{z}-a_kz}(w+\texttt{j})
        \end{pmatrix}
        =
        \begin{pmatrix}
            \eta^0 e^{a_0z'-\overline{a}_0\overline{z}'}\\
            \eta^1 e^{a_1z'-\overline{a}_1\overline{z}'}\\
            \eta^2 e^{a_2z'-\overline{a}_2\overline{z}'}\\
            \eta^3 e^{\overline{a}_k\overline{z}'-a_kz'}(w+\texttt{j})
        \end{pmatrix}\cdot q,
    \]
\end{small}
where $q\in\mathbb{H}$ and $|q|=1$.
But it is easy to find that $q\in\mathbb{C}$.
Since the corresponding components are equal, we know that if $w=0$, then
\begin{small}
    \[
        q=
        \frac{e^{a_0z-\overline{a}_0\overline{z}}}{e^{a_0z'-\overline{a}_0\overline{z}'}}=
        \frac{e^{a_1z-\overline{a}_1\overline{z}}}{e^{a_1z'-\overline{a}_1\overline{z}'}}=
        \frac{e^{a_2z-\overline{a}_2\overline{z}}}{e^{a_2z'-\overline{a}_2\overline{z}'}}.
    \]
\end{small}
Let
\begin{small}
    \begin{sequation}\label{Lambda}
        \Lambda_{\varphi}=\{ z=x+\texttt{i}y\in\mathbb{C} \mid s_jx+(c_j-1)y \equiv 0\pmod{\pi}~(j=1,2) \},
    \end{sequation}
\end{small}
then $\varphi(z)=\varphi(z')$ if and only if $z-z'\in\Lambda_{\varphi}$.
The same argument as in \cite{Liao} shows that $\varphi$ descends to the torus if and only if $rank(\Lambda_{\varphi})=2$, which is equivalent to
\begin{small}
    \begin{sequation}\label{3tc}
        dim_{\mathbb{Q}} span_{\mathbb{Q}} \{ (s_1,c_1-1),~(s_2,c_2-1) \}=2.
    \end{sequation}
\end{small}
Notice that $\theta_1,~\theta_2\in (0,\pi)$ and $\theta_1\neq \theta_2$, then \eqref{3tc} holds and $\varphi$ always descends to the torus now.

If $w\neq 0$, then
\begin{small}
    \[
        q=
        \frac{e^{a_0z-\overline{a}_0\overline{z}}}{e^{a_0z'-\overline{a}_0\overline{z}'}}=
        \frac{e^{a_1z-\overline{a}_1\overline{z}}}{e^{a_1z'-\overline{a}_1\overline{z}'}}=
        \frac{e^{a_2z-\overline{a}_2\overline{z}}}{e^{a_2z'-\overline{a}_2\overline{z}'}}=
        \frac{e^{a_kz-\overline{a}_k\overline{z}}}{e^{a_kz'-\overline{a}_k\overline{z}'}}=
        \overline{\z( \frac{e^{a_kz-\overline{a}_k\overline{z}}}{e^{a_kz'-\overline{a}_k\overline{z}'}} \y)},
    \]
\end{small}
which implies $q\in\mathbb{R}$, but we also have $|q|=1$, so we obtain that
\begin{small}
    \begin{sequation}\label{w-feiling}
        \frac{e^{a_0z-\overline{a}_0\overline{z}}}{e^{a_0z'-\overline{a}_0\overline{z}'}}=
        \frac{e^{a_1z-\overline{a}_1\overline{z}}}{e^{a_1z'-\overline{a}_1\overline{z}'}}=
        \frac{e^{a_2z-\overline{a}_2\overline{z}}}{e^{a_2z'-\overline{a}_2\overline{z}'}}=q=\pm 1.
    \end{sequation}
\end{small}
If $q=1$ in \eqref{w-feiling}, then $e^{a_jz-\bar{a}_j\bar{z}}=e^{a_jz'-\bar{a}_j\bar{z}'}$ which implies that $\text{Im} [a_j(z-z')]\equiv 0\pmod{\pi}~(j=0,1,2)$.
We put $z-z'=x+\texttt{i}y$, then we have $\text{Im} [a_j(z-z')]=s_jx+c_jy=k_j\pi~(k_j\in \mathbb{Z})$.
Let
\begin{small}
    \begin{sequation}\label{Lambda'}
        \Lambda_{\varphi}'=\{ z=x+\texttt{i}y\in\mathbb{C} \mid s_jx+(c_j-1)y \equiv 0\pmod{\pi}~(j=1,2),y \equiv 0\pmod{\pi} \},
    \end{sequation}
\end{small}
then $\varphi(z)=\varphi(z')$ if and only if $z-z'\in \Lambda_{\varphi}'$.
Similarly, if $q=-1$ in \eqref{w-feiling}, we can let
\begin{small}
    \begin{sequation}\label{Lambda''}
        \Lambda_{\varphi}''=\{ z=x+\texttt{i}y\in\mathbb{C} \mid s_jx+(c_j-1)y \equiv 0\pmod{\pi}~(j=1,2),y-\frac{\pi}{2} \equiv 0\pmod{\pi} \},
    \end{sequation}
\end{small}
then $\varphi(z)=\varphi(z')$ if and only if $z-z'\in \Lambda_{\varphi}''$.
We can easily find that $\Lambda_{\varphi}',\Lambda_{\varphi}''\subset \Lambda_{\varphi}$ and $\Lambda_{\varphi}''=\Lambda_{\varphi}'+\frac{\pi}{2}\texttt{i}$, so $0\leq rank(\Lambda_{\varphi}')=rank(\Lambda_{\varphi}'')\leq rank(\Lambda_{\varphi})=2$.
In the following we only compute $rank(\Lambda_{\varphi}')$.
Since $y \equiv 0\pmod{\pi}$ in \eqref{Lambda'}, then we put $y=k\pi~(k\in\mathbb{Z})$, which implies $s_jx \equiv -(c_j-1)k\pi \pmod{\pi}$.
Put $s_jx=m\pi-(c_j-1)k\pi~(m\in\mathbb{Z},j=1,2)$, then
\begin{small}
    \[
        x=\frac{\pi[m-(c_1-1)k]}{s_1}=\frac{\pi[m-(c_2-1)k]}{s_2},
    \]
\end{small}
which shows that for given $\theta_1$ and $\theta_2$, integers $m$ and $k$ must satisfy the following equation:
\begin{small}
    \[
        \frac{m-(c_1-1)k}{s_1}=\frac{m-(c_2-1)k}{s_2}.
    \]
\end{small}
This is equivalent to
\begin{small}
    \begin{sequation}\label{mk-eq}
        (s_1-s_2)m=[s_1(c_2-1)-s_2(c_1-1)]k.
    \end{sequation}
\end{small}
Since $x=\frac{\pi[m-(c_1-1)k]}{s_1}$ and $y=k\pi$, then from \eqref{mk-eq} it is easy to find that $rank(\Lambda_{\varphi}')\leq 1$.
From $rank(\Lambda_{\varphi}'')=rank(\Lambda_{\varphi}')$ we quickly obtain $rank(\Lambda_{\varphi}'')\leq 1$.
The same argument as in \cite{Liao} shows that if $w\neq 0$, then $\varphi$ never descends to the torus.
The above discussion gives us the torus criterion as follows.
\begin{theorem}\label{torus-criterion}
    \eqref{first}, \eqref{second} and \eqref{third} descend to the torus if and only if $w=0$.
\end{theorem}

When $\varphi$ descends to the torus, we immediately obtain the following theorem from Theorem \ref{moduli-space-C} and Theorem \ref{torus-criterion}.
\begin{theorem}\label{moduli-space-tori}
    Let $\mathcal{M}_3(\mathbb{T})$ denote the moduli space of all noncongruent linearly full totally real flat minimal tori in $\mathbb{H}P^3$ that do not lie in $\mathbb{C}P^3$, then $\mathcal{M}_3(\mathbb{T})$ has three components, each of which is a manifold of real dimension $4$ and intersects with two real hypersurfaces at
    \begin{small}
    \begin{equation*}
        \Gamma_3(\mathbb{T})=\z\{ ~(\theta,r)\in \mathbb{R}^2 ~\z|~\frac{\pi}{3}<\theta<\frac{\pi}{2},~0<r<\frac{1}{4\sin^2{\theta}}\y.~ \y\}.
    \end{equation*}
    \end{small}
\end{theorem}

Next, we will compute $\det(A_{FR}^{\prime\varphi})$ when the isotropy order is $2$.
For this, we shall introduce some background knowledge (cf. \cite{U}).
If $\varphi:\mathbb{T}\to \mathbb{H}P^3$ is a harmonic map, then it generates a harmonic sequence in $G(2,8)$, i.e.,
\begin{small}
    \begin{sequation}\label{varphi-harmonic-sequence}
        \underline{\varphi}=\underline{\varphi}_0
        \stackrel{A_{\underline{\varphi}_0,\underline{\varphi}_1}^{\prime}}{\longrightarrow} \underline{\varphi}_{1}
        \stackrel{A_{\underline{\varphi}_1,\underline{\varphi}_2}^{\prime}}{\longrightarrow} \underline{\varphi}_{2}
        \stackrel{A_{\underline{\varphi}_2,\underline{\varphi}_3}^{\prime}}{\longrightarrow}
        \underline{\varphi}_{3}
        \stackrel{A_{\underline{\varphi}_3,\underline{\varphi}_4}^{\prime}}{\longrightarrow}
        \cdots
        \stackrel{A_{\underline{\varphi}_{m-1},\underline{\varphi}_m}^{\prime}}{\longrightarrow} \underline{\varphi}_{m}
        \stackrel{A_{\underline{\varphi}_m,\underline{\varphi}_{m+1}}^{\prime}}{\longrightarrow}
        \cdots.
    \end{sequation}
\end{small}
Moreover we can define a map called the {\it $\partial^{\prime}$-first return map} for $\varphi$ as follows:
\begin{small}
    \[
        A_{FR}^{\prime\varphi}=
        A_{\underline{\varphi}_2,\underline{\varphi}_0}^{\prime}
        \circ
        A_{\underline{\varphi}_1,\underline{\varphi}_2}^{\prime}
        \circ
        A_{\underline{\varphi}_0,\underline{\varphi}_1}^{\prime}:
        C^{\infty}(\underline{\varphi}_0) \to C^{\infty}(\underline{\varphi}_0),
    \]
\end{small}
where
\begin{small}
    \begin{seqnarray}
         A_{\underline{\varphi}_i,\underline{\varphi}_j}^{\prime}:
        C^{\infty}(\underline{\varphi}_i) &\to& C^{\infty}(\underline{\varphi}_j) \nonumber\\
        v&\mapsto &\pi_{\underline{\varphi}_j}(\partial_z v). \nonumber
    \end{seqnarray}
\end{small}

It is clear that the $\partial^{\prime}$-first return map $A_{FR}^{\prime\varphi}$ for $\varphi$ is a linear transformation on $C^{\infty}(\underline{\varphi}_0)$.
In the following we will compute $\det(A_{FR}^{\prime\varphi})$ for \eqref{new-first}, \eqref{new-second} and \eqref{new-third}.
First of all, it is easy to check that if $[s]:\mathbb{C} \to \mathbb{C}P^7$ is a horizontal lift of $\varphi$, then for all $A \in Sp(4)$, $[As]$ is also horizontal.
Hence, sometimes when we compute the harmonic sequence of $\varphi$, we only need to take the partial derivative with respect to $z$ without applying orthogonal projection.
So
\begin{small}
    \[
        \begin{aligned}
            A_{FR}^{\prime\varphi}(s,\texttt{j}s)&=
            A_{\underline{\varphi}_2,\underline{\varphi}_0}^{\prime}
            \circ
            A_{\underline{\varphi}_1,\underline{\varphi}_2}^{\prime}
            \circ
            A_{\underline{\varphi}_0,\underline{\varphi}_1}^{\prime}
            (s,\texttt{j}s)\\
            &=A_{\underline{\varphi}_2,\underline{\varphi}_0}^{\prime}
            (\partial_z^2 s,\partial_z^2 (\texttt{j}s))\\
            &=(s,\texttt{j}s)
            \begin{pmatrix}
                \z\langle \partial_z^3 s,s \y\rangle&\z\langle \partial_z^3 (\texttt{j}s),s \y\rangle\\
                \z\langle \partial_z^3 s,\texttt{j}s \y\rangle&\z\langle \partial_z^3 (\texttt{j}s),\texttt{j}s \y\rangle
            \end{pmatrix}\\
            &=(s,\texttt{j}s)
            \begin{pmatrix}
                \sum_{j=0}^m a_j^3r_j & \sum_{i,j=0}^m \bar{w}_{ij}\xi^i\xi^j a_j^3 e^{(\overline{a}_i+\overline{a}_j)\overline{z}-(a_i+a_j)z} \\
                \sum_{i,j=0}^m w_{ij}\xi^i\xi^j a_j^3 e^{(a_i+a_j)z-(\overline{a}_i+\overline{a}_j)\overline{z}} & (-1)^3\sum_{j=0}^m a_j^3r_j
            \end{pmatrix},
        \end{aligned}
    \]
\end{small}
where $s=UV_0^{(m)}$. Thus
\begin{tiny}
    \begin{seqnarray}
        \det(A_{FR}^{\prime\varphi})&=&
        \begin{vmatrix}
            \sum_{j=0}^m a_j^3r_j & \sum_{i,j=0}^m \bar{w}_{ij}\xi^i\xi^j a_j^3 e^{(\overline{a}_i+\overline{a}_j)\overline{z}-(a_i+a_j)z}\\
            \sum_{i,j=0}^m w_{ij}\xi^i\xi^j a_j^3 e^{(a_i+a_j)z-(\overline{a}_i+\overline{a}_j)\overline{z}} & (-1)^3\sum_{j=0}^m a_j^3r_j
        \end{vmatrix}\nonumber\\
        &=&-\z( \sum_{j=0}^m a_j^3r_j \y)^2-\z( \sum_{i,j=0}^m \bar{w}_{ij}\xi^i\xi^j a_j^3 e^{(\overline{a}_i+\overline{a}_j)\overline{z}-(a_i+a_j)z} \y)\cdot\z( \sum_{i,j=0}^m w_{ij}\xi^i\xi^j a_j^3 e^{(a_i+a_j)z-(\overline{a}_i+\overline{a}_j)\overline{z}} \y).
    \end{seqnarray}
\end{tiny}
So when \eqref{new-first} descends to the torus, by using \eqref{tongjie} and $\theta_1+\theta_2=\pi$ we can compute
\begin{tiny}
    \begin{seqnarray}\label{1-det-FR}
        \det(A_{FR}^{\prime\varphi})
        &=&-\z( \sum_{j=0}^m a_j^3r_j \y)^2-\z( \sum_{i,j=0}^m \bar{w}_{ij}\xi^i\xi^j a_j^3 e^{(\overline{a}_i+\overline{a}_j)\overline{z}-(a_i+a_j)z} \y)\cdot\z( \sum_{i,j=0}^m w_{ij}\xi^i\xi^j a_j^3 e^{(a_i+a_j)z-(\overline{a}_i+\overline{a}_j)\overline{z}} \y)\nonumber\\
        &=&-\z( (r_0-r_3)+a_1^3(r_1-r_4)+a_2^3(r_2-r_5) \y)^2\nonumber\\
        & &-\z( -2\bar{w}_{03}\xi^0\xi^3a_0^3-2\bar{w}_{14}\xi^1\xi^4a_1^3-2\bar{w}_{25}\xi^2\xi^5a_2^3 \y)
        \cdot
        \z( -2w_{03}\xi^0\xi^3a_0^3-2w_{14}\xi^1\xi^4a_1^3-2w_{25}\xi^2\xi^5a_2^3 \y)\nonumber\\
        &=&-\z( (r_0-r_3)+a_1^3(r_1-r_4)+a_2^3(r_2-r_5) \y)^2-4|w_{03}|^2r_0r_3\z( a_0^3+\frac{\sin\theta_2}{\sin(\theta_1-\theta_2)}a_1^3+\frac{\sin\theta_1}{\sin(\theta_2-\theta_1)}a_2^3 \y)^2\nonumber\\
        &=&-\z[(r_0-r_3)^2+4|w_{03}|^2r_0r_3\y]\cdot\z( 1-\frac{1}{2c_1}a_1^3-\frac{1}{2c_1}\bar{a}_1^3 \y)^2.
    \end{seqnarray}
\end{tiny}
Notice that the first factor of \eqref{1-det-FR} does not equal to $0$, and easy computation shows that the second factor of \eqref{1-det-FR} does not equal to $0$ too.
This implies that the isotropy order of \eqref{new-first} is just $2$.
By performing similar computations on \eqref{new-second} and \eqref{new-third}, we obtain
\begin{small}
    \begin{seqnarray}\label{2-det-FR}
        \det(A_{FR}^{\prime\varphi})
        &=&-\z[(r_0-r_3)^2+4|w_{03}|^2r_0r_3\y]\cdot\z( 1-2c_1a_1^3+a_1^6 \y)^2
    \end{seqnarray}
\end{small}
and
\begin{small}
    \begin{seqnarray}\label{3-det-FR}
        \det(A_{FR}^{\prime\varphi})
        &=&-\z[ (r_0-r_3)^2+4|w_{03}|^2r_0r_3 \y]\cdot\z( 1-2c_2a_2^3+a_2^6 \y)^2
    \end{seqnarray}
\end{small}
respectively.
Then we consider the following equation
\begin{small}
    \begin{sequation}\label{6-order-equ}
        a^6-2ca^3+1=0,
    \end{sequation}
\end{small}
where $a=e^{\texttt{i}\theta}=c+\texttt{i}s,~c=\cos{\theta},~s=\sin{\theta},~\frac{\pi}{3}<\theta<\frac{\pi}{2}$ or $\frac{\pi}{2}<\theta<\frac{2\pi}{3}$.
Put $t=a^3$, then \eqref{6-order-equ} becomes $t^2-2ct+1=0$, and the complex roots of this equation given by the quadratic formula are as follows:
\begin{small}
    \[
        t_{1,2}=\frac{2c\pm\sqrt{-4s^2}}{2}=c\pm \texttt{i}s=e^{\pm \texttt{i}\theta}.
    \]
\end{small}
Since $t=a^3=e^{\texttt{i}\cdot 3\theta}$, together with the range of $\theta$, if the above equation has roots, then we must have $3\theta=2\pi-\theta$, so $\theta=\frac{\pi}{2}$ that does not lie in the range of $\theta$.
Therefore \eqref{6-order-equ} has no roots.
As a result, $\det(A_{FR}^{\prime\varphi})\neq 0$ for \eqref{new-second} and \eqref{new-third}.
When \eqref{new-first}, \eqref{new-second} and \eqref{new-third} descend to the torus, we always have $\det(A_{FR}^{\prime\varphi})\neq 0$, by \cite{U} we quickly obtain the following theorem.
\begin{theorem}
\label{isotropy-order-2}
    When \eqref{new-first}, \eqref{new-second} and \eqref{new-third} descend to the torus, each is covered by a primitive harmonic map of finite type into $SU(8)/S(2U(2) \times U(4))$.
\end{theorem}
\begin{proof}
    Since $\det(A_{FR}^{\prime\varphi})\neq 0$, then by the last theorem of \cite{U} we know that each $\varphi$ of type \eqref{new-first}, \eqref{new-second} or \eqref{new-third} is covered by a primitive harmonic map of finite type into $SU(8)/S(2U(2) \times U(4))$.
\end{proof}

\begin{remark}
    If $\varphi$ is a linearly full totally real flat minimal surface of isotropy order $2$ in $\mathbb{H}P^3$ that lies in $\mathbb{C}P^3$ up to $Sp(4)$, then we can denote $\varphi$ by $[s]$ and compute $\det(A_{FR}^{\prime\varphi})$ as follows:
    \begin{small}
        \[
            A_{FR}^{\prime\varphi}(s)=
            (s) \cdot
            \z\langle \partial_z^3 s,s \y\rangle.
        \]
    \end{small}
    Since the isotropy order of $\varphi$ is $2$, we have $\det(A_{FR}^{\prime\varphi})=\z\langle \partial_z^3 s,s \y\rangle \neq 0$.
    Thus for all linearly full totally real flat minimal surfaces of isotropy order $2$ in $\mathbb{H}P^3$, we always have $\det(A_{FR}^{\prime\varphi}) \neq 0$.
    For the general target $\mathbb{H}P^n$, if the isotropy order $\mathrm{r}$ is odd, then it follows from Sec.3 in \cite{U} that $\det(A_{FR}^{\prime\varphi}) \neq 0$.
    Is it true in the case of even isotropy order?
\end{remark}

\end{document}